\documentclass[12pt,oneside,a4]{amsart}
\usepackage{setspace}
\usepackage{amsfonts}
\usepackage{amsmath}
\usepackage{amssymb}
\usepackage{amsthm}
\usepackage{mathrsfs}
\usepackage[all]{xy}

\newcommand{\proofend}{\hfill \hbox{\vrule width 5pt height 5pt depth
0pt}}
\newcommand{\Qtp}{\mathbb{Q}^{\mathrm{t}.p}}

\newcommand{\C}{\mathbb{C}}
\newcommand{\Z}{\mathbb{Z}}
\newcommand{\N}{\mathbb{N}}
\newcommand{\Q}{\mathbb{Q}}

\newcommand{\proj}{\mathbb{P}}

\newcommand{\F}{\mathbb{F}}

\newtheorem{thm}{Theorem}

\newtheorem{lemma}{Lemma}[section]
\newtheorem{corol}[thm]{Corollary}

\newtheorem{propo}[lemma]{Proposition}

\begin{document}

\title[A proof of the Schinzel-Zassenhaus conjecture]{A proof of the Schinzel-Zassenhaus conjecture on  polynomials}

\begin{abstract}
We prove that if $P(X) \in \Z[X]$ is an integer polynomial of degree $n$ and having $P(0) = 1$, then either $P(X)$ is a product of cyclotomic polynomials, or else at least one of the complex roots of $P$ belongs to the disk $|z| \leq 2^{ - 1 / (4n) }$. We also obtain a relative version of this result over the compositum $\Q^{\mathrm{ab}} \cdot \Qtp$ of all abelian
and all totally $p$-adic extensions of $\Q$, for any fixed prime~$p$, and apply it to prove a $\Q^{\mathrm{ab}} \cdot \Qtp$-relative
canonical height lower bound on the multiplicative group. Another extension is given to a uniform positive height lower bound, inverse-proportional to the total number of singular points,
on holonomic power series in $\Q[[X]]$ and not of the form $p(X) / (X^k-1)^m$, where $p(X) \in \Q[X]$, with 
a further application to existence of a small critical value for certain rational functions. 
\end{abstract}

\author{Vesselin Dimitrov}

\address{Department of Mathematics \\ University of Toronto \\ 40 St George Street \\
Toronto ON, M5S 2E5}

\email{vesselin.dimitrov@gmail.com; dimitrov@math.toronto.edu}

\maketitle

\section{Introduction and main results}

In this paper we propose a solution of the Schinzel-Zassenhaus conjecture on integer polynomials~\cite{schzass}. The method, approximately speaking, is to blend two distinct traditions in number theory: the arithmetic algebraization theorem of Carlson, P\'olya and Bertrandias (cf. Amice~\cite{amice}, \emph{Th\'eor\`eme 5.4.6}, and Chambert-Loir~\cite{algebraicite} for an overview of this subject); and a use of an Euler-Frobenius congruence at a single and fixed prime~$p$, rather in the style of Amoroso and Dvornicich~\cite{amorosodvornicich} (cf. Deligne~\cite{deligne} for the most general such congruence in the unramified case, in the form of the Frobenius-twisted identity $\mathrm{tr}(A^{p^n}) = F(\mathrm{tr}(A^{p^{n-1}}))$ for a square matrix $A$ over the ring of length-$n$ Witt vectors of an $\F_p$-algebra).

Our result is completely explicit:

\begin{thm}  \label{original}
  Let $P \in \Z[X]$ be a monic integer irreducible polynomial of degree $n > 1$. If $P$ is not cyclotomic, then
  $$
  \max_{\alpha: \, P(\alpha) = 0} |\alpha| \geq 2^{ \frac{1}{4n} } = 1 + \frac{\log{2}}{4n} + O(1/n^2).
  $$
\end{thm}

Reflecting the $\{p, \infty\}$-adic nature of the method, our result extends in a relative form over the compositum $\Q^{\mathrm{ab}} \cdot \Qtp$ in $\C$ of all abelian extensions of $\Q$ and all
finite Galois extensions $K/\Q$ in which a fixed prime $p$ splits completely. Throughout the paper, let us denote by $O_p$ the ring of algebraic integers of this infinite Galois
extension $\Q^{\mathrm{ab}} \cdot \Q^{\mathrm{t.p}}$ of $\Q$.

\begin{thm} \label{sz}
Let   $P(X) \in (\mathbb{Q}^{\mathrm{ab}} \cdot \Qtp)[X]$ be an irreducible (over $\mathbb{Q}^{\mathrm{ab}} \cdot \Qtp$) polynomial of degree $n > 1$. With $K \subset \Q^{\mathrm{ab}} \cdot \Qtp$ a finite subextension of $(\Q^{\mathrm{ab}} \cdot \Qtp) / \Q$ containing the coefficients of $P$,  define $c_{P/K} \in \N$ to be the leading coefficient of the element of $\Z[X]$ proportional to $N_{K/\Q}(P)$ and without common prime factor for its coefficients, with the convention $c_{P/K} > 0$. (Thus $c_{P/K} \leq |N_{K/\Q}(a_0)|$ if $P(X) \in O_p[X]$ with
leading coefficient $a_0 \in O_p$, and in particular, $c_{P/K} = 1$ in the case of a monic $P(X) \in O_p[X]$.)

 Then
\begin{eqnarray*} \label{inab}
\frac{p^3}{[K:\Q]} \log{c_{P/K}} +  \frac{2p^2}{[K:\Q]} \sum_{ \sigma : K \hookrightarrow \C
  } \max_{\alpha: \, \sigma(P)(\alpha) = 0} \log^+{|\alpha|}  \geq \frac{\log{2}}{n}.
\end{eqnarray*}
In the case that $p$ is unramified in $F$, the inequality strengthens to
\begin{eqnarray*}
\frac{1}{2[K:\Q]} \log{c_{P/K}} +  \frac{1}{[K:\Q]} \sum_{ \sigma : K \hookrightarrow \C
  } \max_{\alpha: \, \sigma(P)(\alpha) = 0} \log^+{|\alpha|}  \geq \frac{\log{2}}{2p^2n}.
\end{eqnarray*}
Lastly, in the case that $P(X) \in \Qtp[X]$, it strengthens further to
\begin{eqnarray*}
\frac{1}{[K:\Q]} \log{c_P} +  \frac{1}{[K:\Q]} \sum_{ \sigma : K \hookrightarrow \C
  } \max_{\alpha: \, \sigma(P)(\alpha) = 0} \log^+{|\alpha|} \\
  - \frac{1}{p}\frac{1}{[K:\Q]} \sum_{ \sigma : K \hookrightarrow \C
  } \min_{\alpha: \, \sigma(P)(\alpha) = 0} \log^+{|1/\alpha|}  \geq \frac{\log{2}}{p^2n}.
\end{eqnarray*}
\end{thm}

The proof of Theorem~\ref{original}, given on pages~15--16, requires only a relatively
small part 
of the preparatory material in the paper: Lemma~\ref{mod4}, Proposition~\ref{squarerad}, 
Lemma~\ref{square} and Corollary~\ref{rationalinteger}.
If $P^{\star}(X) := X^n P(1/X)$ is the reciprocal polynomial, the
monicity condition translates to $P^{\star}(0) = 1$, and
  Theorem~\ref{original} is equivalent to the dichotomy that either $P^{\star}(X)$ is
cyclotomic, or else $P^{\star}(z)$ has at least one complex root belonging to the disk $|z| \leq 2^{-\frac{1}{4n}}$.

 Our approach in a nutshell
consists of
 expanding
\begin{equation} \label{congruence}
f(1/X) := \sqrt{\prod_{\alpha: \, P(\alpha) = 0}(1 - \alpha^2 / X)(1 - \alpha^{4} / X)} \in \Z[[1/X]]
\end{equation}
in a formal power series about $\infty \in \proj^1$,  verifying that the coefficients in this expansion are integers, and applying
Bertrandias's theorem to the power series $f(1/X)$ and the ``hedgehog'' $\mathcal{K} = \mathcal{K}_f$ consisting of the radial line segments joining $0$ to the set of $2n$ (or fewer) roots $\{ \alpha^2, \alpha^4 \mid P(\alpha) = 0 \}$. It turns out that this hedgehog $\mathcal{K}$ has transfinite diameter less than $1$ if $\max_{\alpha: \, P(\alpha) = 0} |\alpha| < 2^{\frac{1}{4n}}$.


\medskip

{\it Acknowledgements.} I am grateful to Emmanuel Breuillard and P\'eter Varj\'u at Cambridge DPMMS,
and Jacob Tsimerman at the University of Toronto, for
 many inspiring discussions on
the subject of Lehmer's problem, and for their valuable feedback  clarifying the scope of the method proposed in this paper. The key idea of using a Diophantine
rationality theorem in the context of the Schinzel-Zassenhaus conjecture was inspired by the recent
 paper~\cite{bnz} by
Jason Bell, Khoa Nguyen and Umberto Zannier, specifically upon trying to make a further advance
on their Theorem~3.3 (on non-rational holonomic power series) in the univariate case. The result
of this enquiry, and our starting point for the present paper, comes down to Theorem~\ref{holonomic} below, and I
heartily
thank  the authors of~\cite{bnz}  for answering some of my questions on their work, and for an illuminating discussion.
The reference to the MathOverflow discussion~\cite{mo},
and in particular to Dubinin's work, was kindly pointed out to me by Alexander Er\"emenko and Fedor Nazarov.
I thank Jason Bell, Yuri Bilu, Pierre Deligne,  Art$\bar{\mathrm{u}}$ras Dubickas,  Philipp Habegger, Lars K\"uhne,
 Curt McMullen 
and Khoa Nguyen for their comments on a first draft of this paper.

Part of this research was done at Cambridge DPMMS. The author gratefully acknowledges support from the European Research Council via ERC grant GeTeMo~617129.

\section{The $p$-adic input: a mod $p^2$ congruence and integrality of the coefficients of a certain algebraic power series} \label{padic}

In this section we explain the key integrality property (\ref{congruence}) (case $p = 2$), and its generalization to an arbitrary prime $p$. This is the $p$-adic input of our approach. We begin with the simplest case, which suffices for the proof of Theorem~\ref{original}, and follow it up
with the generalization. The following is due independently to Smyth~\cite{smyth} and Arnold~\cite{arnoldeuler,arnold}. We include the short proof
for the reader's convenience.

\begin{lemma}[Arnold~\cite{arnoldeuler,arnold}, Smyth~\cite{smyth}]  \label{mod4}
  Consider $P(X) \in \Z[X]$ an integer monic polynomial, with factorization $P(X) = \prod_{i=1}^n (X - \alpha_i)$ over
  $\C$. Define the ancillary sequence of integer polynomials $(P_m)_{m \in \N}$ by
  $$
  P_m(X) := \prod_{i=1}^n (X - \alpha_i^m) \in \Z[X].
  $$
  Then the following mod $4$ congruence holds:
  \begin{equation}\label{modph}
    P_{4} \equiv P_2 \mod{4 \, \Z[X]}.
  \end{equation}
\end{lemma}

\begin{proof}
  The simplest way to see it is by the following symmetric powers argument, told to me by Jacob Tsimerman. First off,
  we easily see the congruence $\sum_{i=1}^n \alpha_i^4 \equiv \sum_{i=1}^n \alpha_i^2 \mod{4}$ for the $X^{n-1}$ terms. For, by the
  Newton-Girard
  identities expressing the power sums $s_m := \sum_{i=1}^n \alpha_i^m$ as polynomials in the $n$ elementary symmetric functions $e_1, \ldots,
  e_n \in \Z$ of the roots $\alpha_i$, we have
\begin{eqnarray*}
  s_2 = e_1^2 - 2e_2, \\
  s_4 = e_1^4 +2e_2^2 -4(e_1^2e_2-e_1e_3+e_4),
\end{eqnarray*}
which are manifestly congruent mod $4$ as $e_1, e_2, e_3, e_4 \in \Z$.

For the general case with the $X^{n-k}$ coefficients, we simply apply the previous to the
new ``$k$-th symmetric power'' polynomial
$$
\prod_{I} (X - \alpha_{i_1} \cdots \alpha_{i_k}) \in \Z[X],
$$
 where
the product is over all $k$-element sets $I = \{i_1, \ldots, i_k\} \subset \{1, \ldots,n\}$.
\end{proof}

The following consequence is the crucial $2$-adic integrality property needed for our proof
of Theorem~\ref{original}. It is based on the observation that
\begin{equation} \label{2int}
\sqrt{1 + 4Y} = \sum_{k =0}^{\infty}
\binom{1/2}{k} 4^kY^k \in \Z[[Y]].
\end{equation}

\begin{propo} \label{squarerad}
  Consider a polynomial $Q \in \Z[X]$ such that $Q(0) = 1$ and $Q \mod{4}$ is a perfect
  square in $(\Z/4)[X]$. Then the square radical $\sqrt{Q(X)}$ has an $X = 0$ Taylor expansion
  with integer coefficients: $\sqrt{Q(X)} \in 1 + X \Z[[X]]$.

  In particular, for every integer  polynomial $P \in \Z[X]$ having $P(0) = 1$, we have
  \begin{equation}  \label{radical}
  \sqrt{P_2(X)P_4(X)} \in \Z[[X]].
  \end{equation}
\end{propo}

\begin{proof}
  By assumption, we can write $Q(X) = U(X)^2 + 4V(X)$, with $U(X), V(X) \in \Z[X]$
  and $U(0) = 1, V(0) = 0$. We have
  $$
  \sqrt{Q(X)} = U(X) \sqrt{ 1 + 4 V(X) / U(X)^2 }  \in \Z[[X]],
  $$
  upon applying (\ref{2int}) with $Y := V(X) / U(X)^2 \in X\Z[[X]]$.
\end{proof}

{\it Examples. } Interesting examples of algebraic or holonomic power series
with integer coefficients are often found among the generating functions enumerating
certain lattice walks, cf.~\cite{bostan,bousquet}.
The simplest, and most celebrated cases are the Catalan, Schr\"oder and
Motzkin numbers, of respective generating functions
$$
\frac{1 - \sqrt{1-4x}}{2x}, \quad \frac{1 -  x - \sqrt{1-6x + x^2}}{2x},
$$
and
$$
 \frac{1 -  x - \sqrt{1-2x -3x^2}}{2x^2}.
$$
In the light of Proposition~\ref{squarerad}, we see
a purely algebraic reason that these quadratic algebraic functions happen
to have integer coefficients expansions:
 the expressions under the square root are congruent modulo $4$
to a perfect square: respectively, $1, (x-1)^2$ and $(x-1)^2$. \proofend

\medskip

Next, we figure out the cases that the function~(\ref{radical}) is rational,
i.e. that the polynomial $P_2P_4$ is a perfect square. This is the place in the
argument where the cyclotomic cases are recognized and excluded.

\begin{lemma} \label{square}
  Suppose that $P(X) \in \Z[X]$ is a monic integer irreducible polynomial
  of $\deg{P} > 1$, such that $P_2(X)$ is not
 a perfect square. Then the following are equivalent:
  \begin{itemize}
    \item[(i)] $P$ is a cyclotomic polynomial of an odd level;
    \item[(ii)] $P_2 = P_4$;
    \item[(iii)] $P_2P_4$ is a perfect square;
    \item[(iv)] the function $\sqrt{P_2P_4}$ is rational.
    \end{itemize}
  \end{lemma}

\begin{proof}
  If $P = \Phi_N$ with $N$ odd, it is plain that $P_4 = P_{2} = \Phi_N$ also, and hence
$\sqrt{P_2P_4} = \Phi_N$.

Reversely, suppose that $P_2P_4$ is a perfect square,
and take any root $\alpha$ of our irreducible polynomial $P$. Since we assume the polynomial $P_2 \in \Z[X]$
to be irreducible,  $\alpha^2$ has to be among the roots $\sigma(\alpha)^{4}$
of $P_4$, with $\sigma$ ranging over the Galois group of the polynomial $P$. Letting $k$ the order of $\sigma$, we get
by a $k$-fold iteration:
\begin{eqnarray*}
\alpha^2 = \sigma(\alpha^2)^{2} = \sigma^2(\alpha)^{2^3} = \sigma^2( \alpha^2)^{2^2}
= \cdots = \sigma^k(\alpha)^{2^{k+1}} = \alpha^{2^{k+1}}.
\end{eqnarray*}
Hence $\alpha$ is a root of unity.

An alternative argument, based on the theory of the Weil height, is to note that $\alpha^2 = \sigma(\alpha^{4})$ implies
that $2h(\alpha) = h(\alpha^2) = h(\alpha^{4}) = 4h(\alpha)$, hence $h(\alpha) = 0$, and once again $\alpha$
is a root of unity.
\end{proof}

In the remainder of this section, we extend our preparatory results so far to the relative
setting over $\Q^{\mathrm{ab}} \cdot \Qtp$. The reader interested only
in the proof of Theorem~\ref{original} can skip directly to section~\ref{hedgehogs}.

We begin by characterizing the ring of algebraic $p$-integers in $\Q^{\mathrm{ab}} \cdot \Qtp$: 
the elements that fall in the valuation ring 
$$
O_{\C_p} = \{ x \in \C_p \mid |x|_p \leq 1 \}
$$
under every field embedding $\Q^{\mathrm{ab}} \cdot \Qtp \hookrightarrow \C_p$ in the completion 
$\C_p$ of an algebraic closure of $\Q_p$. Let $R_p \subset \Qtp$ be the subring
of algebraic $p$-integers in $\Qtp$.

\begin{lemma} \label{pintegers}
  The algebraic $p$-integers in $\Q^{\mathrm{ab}} \cdot \Qtp$ consist precisely of
  the elements of the form
  $$
  \eta = f(\zeta), \quad \textrm{where } f(X) \in R_p[X] \textrm{ and } \zeta \in \mu_{\infty} \textrm{ is a root of unity}.
  $$ 
\end{lemma}

\begin{proof}
Clearly these elements are algebraic $p$-integers in $\Q^{\mathrm{ab}} \cdot \Qtp$. 
For the converse, let $\eta \in \Q^{\mathrm{ab}} \cdot \Qtp$ be a $p$-integral
element, and consider first the case that $p$ does not ramify in $\Q(\eta)$. 
By the Kronecker-Weber theorem, there is then a root of unity $\zeta_N$ 
of an order $N \not\equiv 0\mod{p}$, and a unique polynomial $f(X) \in \Qtp[X]$ of 
$\deg{f} < [\Qtp(\zeta_N):\Qtp]$, such
that $\eta = f(\zeta_N)$. We need to see that the coefficients of $f(X)$ belong to 
the subring $R_p$ of algebraic $p$-integers in $\Qtp$. This follows upon noting that the $\mathrm{Gal}((C_N \cdot \Qtp) / \Qtp)$-conjugates
of $\eta$ are expressed as a linear system in the coefficients of $f$, with $N$-th roots of unity
as the coefficients of the linear system. As $\prod_{\zeta \in \mu_N \setminus \{1\}} (1-\zeta) = N$, 
this  non-vanishing square determinant has an inverse in $R_p[1/N]$.
The latter ring is equal to $R_p$ as $p \nmid N$ in our 
case under consideration, and Cramer's formula expresses the
coefficients of $f(X)$ as $p$-integral linear combinations
of the $\mathrm{Gal}((C_N \cdot \Qtp) / \Qtp)$-conjugates
of $\eta$, which in turn we assume to be themselves algebraic $p$-integers. Hence
$f(X) \in R_p[X]$, finishing the unramified case.

For the general case, the Kronecker-Weber theorem implies again that 
$\eta \in C_N \cdot \Qtp$ for some $N$, where $C_N := \Q(\mu_N)$ 
is the cyclotomic field of level $N$.
  Writing $N = p^f M$ with $p \nmid M$, let us consider $g(X) \in (C_M \cdot \Qtp)[X]$
the unique polynomial of  degree satisfying
$$
\deg{g} < [C_N\cdot \Qtp : C_M \cdot \Qtp] =  \#(\Z / p^f \Z)^{\times} = p^{f-1}(p-1)
$$
 such that
$\eta = g(\zeta_{p^f} - 1)$, where $\zeta_{p^f} \in \mu_{p^f}$ is a primitive
root of unity of order $p^f$. As $(p) = (\zeta_{p^f} - 1)^{p^{f-1}(p-1)}$
in $O_{C_{p^f}}$, we have 
$$
\mathrm{val}_p(\zeta_{p^f} - 1) = 1 / (p^{f-1}(p-1))
$$
 for the 
$p$-adic
valuation normalized by $\mathrm{val}_p(p) = 1$. Under any embedding of $\Q^{\mathrm{ab}} \cdot \Qtp \hookrightarrow \C_p$
the valuation of $\eta$ is non-negative while the valuations of the coefficients of $g$ are rational integers, and
comparing valuations in the equation $\eta = g(\zeta_{p^f} - 1)$ shows that $g$ has $p$-integral coefficients
in $C_M \cdot \Qtp$. Now $p \nmid M$, and we are reduced to the previous (unramified) case.
\end{proof}

We turn now to the generalization of Lemma~\ref{mod4}, relative over $\Q^{\mathrm{ab}} \cdot \Qtp$. For polynomials over $\Z$, several
different proofs can be found in  Smyth~\cite{smyth}, Arnold~\cite{arnoldeuler},
Zarelua~\cite{zarelua,zareluasurvey}, Vinberg~\cite{vinberg}, Deligne~\cite{deligne}
and Dubickas~\cite{dubickascong}.
Those proofs also apply relatively over a totally $p$-adic field. We explain now how to derive our requisite
$\Q^{\mathrm{ab}} \cdot \Qtp$-relative congruence
from the results in~\cite{vinberg}. Extending the notation of Lemma~\ref{mod4} to the
non-monic case, we define
$$
P(X) = a_0\prod_{i=1}^n (X- \alpha_i) \quad \textrm{and} \quad P_m(X) :=
a_0^{m}\prod_{i=1}^n (X - \alpha_i^m),
$$
so that $P_m \in O_p[X]$ whenever $P \in O_p[X]$.
The following ramification dichotomy is
in the style of Amoroso and Dvornicich~\cite{amorosodvornicich}. We may
call the Galois automorphism $\kappa$ here the \emph{Amoroso-Dvornicich element}.
In the unramified case (i), this element is canonical and coincides with the unique Frobenius element at $p$; compare
to Deligne~\cite{deligne}, of which (i) with $\kappa = \mathrm{Frob}_p$ is a
very  particular case. In the ramified case (ii), however, $\kappa$
is non-canonical and amounts to a choice of a generator
for the cyclic group $\mathrm{Gal}(C_N / C_{N/p}) \cong (\Z / p)^{\times}$.

\begin{lemma} \label{cong}
  For every finite Galois subextension $F / \Q$ of $\Q^{\mathrm{ab}} \cdot \Qtp / \Q$,
  there exists a field automorphism $\kappa \in \mathrm{Gal}(F / \Q)$
  such that the following congruence holds true.

  Let $P(X) \in O_F[X]$ be a polynomial with coefficients belonging to the ring of integers $O_F = F \cap O_p$
   of $F$, with leading coefficient $a_0 \in O_F \setminus \{0\}$, and irreducible over $O_p$.
   \begin{itemize}
   \item[(i)] If $p$ is unramified in $F$, then
  \begin{equation}\label{modph}
    P_{p^2} \equiv P_p^{\kappa} \mod{p^2 \cdot O_p[X]};
  \end{equation}
  \item[(ii)] If $p$ is ramified in $F$, so that $p \mid N$ for the minimal level $N$ of a cyclotomic field
  $C_N := \Q(\exp(2 \pi i / N))$ with $F \subset C_N \cdot \Q^{\mathrm{t}.p}$, then
  \begin{equation} \label{ramified}
 P_{p^2} \equiv P_{p^2}^{\kappa} \mod{p^2 \cdot O_p[X]},
  \end{equation}
  and moreover, if $P_{p^2}^{\kappa}$ and $P_{p^2}$ are constant multiples of one another,
  then there exists a root of unity $\zeta \in \mu_{Np}$
  such that $(a_0\zeta^n)^{-1} P(\zeta X) \in (C_{N/p} \cdot \Qtp)[X]$.
  \end{itemize}
\end{lemma}

\begin{proof}
Let us first consider the case that $p$ is unramified in $F$.
Thus there exists a cyclotomic field $C_N$ with $p \nmid N$, such that $P(X) \in (C_N \cdot \Q^{\mathrm{t.p}})[X]$. In this case
we define $\kappa|_{C_N} := \mathrm{Frob}_p = \langle p \rangle$ to be the Frobenius element of $\mathrm{Gal}(C_N / \Q) \cong (\Z / N)^{\times}$,
raising the roots of unity $\zeta_N \mapsto \zeta_N^p$ to the $p$-th power. We define $\kappa \in \mathrm{Gal}(F/\Q)$ by extending $\kappa$ to be the
identity on the totally $p$-adic field. In other words $\kappa \in \mathrm{Gal}(F/\Q)$ is the unique member of the Frobenius conjugacy class at $p$ of the finite Galois extension $F/\Q$ unramified over $p$; this Frobenius conjugacy class is central (a singleton) by our assumption that $F$ is contained in the composite of an abelian and totally $p$-adic Galois extensions of $\Q$.  We verify the congruence (\ref{modph}) for this Galois element $\kappa$.

As in the proof of Lemma~\ref{mod4}, the $k$-th symmetric powers trick
reduces our task to verifying the mod $p^2$ congruence
  between two elements of the integer ring $O_p$ of $\Q^{\mathrm{ab}} \cdot \Qtp$:
  \begin{equation} \label{free}
  \sum_{i=1}^n \alpha_i^{p^2} \equiv \kappa \Big( \sum_{i=1}^n \alpha_i^p \Big) \mod{p^2 \cdot O_p},
  \end{equation}
  provided that all the elementary symmetric functions in the $\alpha_i$ are in~$O_F$.
 This congruence is a special case of Deligne's identity~\cite{deligne} applied to $n = 2$ (of \emph{loc. cit.}) and
 $R = O_p / \mathfrak{a}$ ranging over all characteristic $p$ quotient rings of $O_p$.

  Let us indicate also how to derive (\ref{free}) from the results in Vinberg~\cite{vinberg}.
We start with the basic Frobenius congruence
\begin{equation} \label{frob}
 \Big( \sum_{i=1}^n \alpha_i \Big)^p \equiv \kappa \Big( \sum_{i=1}^n \alpha_i \Big) \mod {p \cdot O_p},
\end{equation}
defining our choice of the Galois element $\kappa \in \mathrm{Gal}(F/\Q)$.
Raising it to the power $p$, using that $(c + pA)^p \equiv c^p \mod{p^2\cdot \Z[A]}$, we get:
\begin{equation}\label{modp2}
\Big( \sum_{i=1}^n \alpha_i \Big)^{p^2} \equiv \kappa  \Big(   \sum_{i=1}^n \alpha_i \Big)^p \mod{p^2 \cdot O_p}.
\end{equation}
The  congruence (\ref{free}) follows now upon expanding the multinomials in (\ref{modp2}) and using
Lemma~1 and Theorem~2 (on the multinomial coefficients) in Vinberg~\cite{vinberg}. See also the congruence (9) of {\it loc. cit.} to cover the cases $p = 2$ and $p = 3$. This concludes the proof of the unramified case and the congruence~(\ref{modph}).

Consider now the case that $p$ ramifies in $F$, so that $p \mid N$ for the minimal level $N$ of a cyclotomic field $C_N$
having $F \subset C_N \cdot \Q^{\mathrm{t.}p}$. In this case, following the lead from Amoroso and Dvornicich~\cite{amorosodvornicich},
we firstly define $\kappa \in \mathrm{Gal}(C_{Np} / \Q)$ to be a generator for the cyclic Galois group $\mathrm{Gal}(C_{Np} / C_{N / p})$, and extend it by the identity on the totally $p$-adic field. Concretely, we pick choices $\zeta_{Np} \in \mu_{Np}$ and $\lambda \in \mu_{p^2}$ of primitive roots of unity of respective orders $Np$ and $p^2$,
and define $\kappa(\zeta_{pN}) := \lambda \zeta_{pN}$ and
\begin{equation} \label{fermat}
\kappa\big( \sum_{j = 0}^{Np-1}  a_j \zeta_{Np}^j \big) := \sum_{j = 0}^{Np-1} a_j \lambda^j \zeta_{Np}^j \textrm{ for } a_j \in \Q^{\mathrm{t.}p}.
\end{equation}
This restricts as a field automorphism of $F/\Q$. If the polynomials $P_{p^2}$ and $P_{p^2}^{\kappa}$
are scalar-proportional, then for any root $\alpha = \alpha_i$ of $P$ there exists a root of unity $\xi = \lambda^r \in \mu_{p^2}$ of order dividing $p^2$ and such that $\xi \alpha$
is among the roots of $P^{\kappa}$. Choosing $\zeta := \zeta_{Np}^{r} \in \mu_{Np}$, we have $\kappa(\zeta) = \lambda^{r} \zeta = \xi \zeta$,
and hence $\zeta^{-1} \alpha$ is a root of both $P(\zeta X)$ and $P(\zeta X)^{\kappa}$. As we assume that $P(X)$,
and hence (equivalently) $P(\zeta X) \in O_p[X]$, is irreducible over $O_p$, it follows
that the polynomials $(a_0\zeta^n)^{-1} P(\zeta X) = ( (a_0\zeta^n)^{-1} P(\zeta X) )^{\kappa}$.
This means that $(a_0\zeta^n)^{-1} P(\zeta X) \in (C_{N/p} \cdot \Qtp)[X]$.

Let us verify the congruence~(\ref{ramified}) for this Galois element $\kappa \in \mathrm{Gal}(F/\Q)$. By the $k$-th symmetric
power trick, we are once again reduced to establishing
  \begin{equation} \label{freebis}
  \sum_{i=1}^n \alpha_i^{p^2} \equiv \kappa \Big( \sum_{i=1}^n \alpha_i^{p^2} \Big) \mod{p^2 \cdot O_p}
  \end{equation}
  whenever the elementary symmetric functions of the $\alpha_i$ belong to $O_F$. We start as before with the more 
  basic mod $p$ congruence
  \begin{equation}\label{basicmodp}
    \Big( \sum_{i=1}^n \alpha_i \Big)^p \equiv  \kappa \Big( \sum_{i=1}^n \alpha_i \Big)^p,
  \end{equation}
  which we prove (cf.~the argument for Lemma~2 (2) in~\cite{amorosodvornicich}) by
  writing out $\eta := \sum_{i=1}^n \alpha_i = f(\zeta_N)$ via Lemma~\ref{pintegers}, with $f(X) \in R_p[X]$, 
 and specializing the Fermat congruence 
  $f(X)^p \equiv f(X^p) \mod{p \cdot R_p[X]}$  to $X = \zeta_N$. This gives the first congruence
  as $\kappa(\eta^p) = \kappa(f(\zeta_N)^p) \equiv \kappa(f(\zeta_N^p)) \equiv f(\kappa(\zeta_N^p)) =
  f(\zeta_N^p) \equiv f(\zeta_N)^p = \eta^p \mod{p \cdot R_p}$ by the definition (\ref{fermat}) of $\kappa$.
  
  From this, noting $(c + pA)^p \equiv c^p \mod{p^2 \cdot \Z[A]}$, we similarly obtain
  \begin{equation}\label{raised}
    \Big( \sum_{i=1}^n \alpha_i \Big)^{p^2} \equiv \kappa \Big( \sum_{i=1}^n \alpha_i \Big)^{p^2} \mod{p^2 \cdot O_p}.
  \end{equation}
  The result~(\ref{freebis}) now similarly follows from Vinberg's multinomial coefficients congruences~\cite{vinberg}.
\end{proof}

As a consequence of this congruence, we derive an integrality result for the Taylor
series of the algebraic functions of the form $\sqrt[p]{P_p^{p-1} P_{p^2}^{\kappa}}$, respectively $\sqrt[p]{P_{p^2}^{p-1} P_{p^2}^{\kappa}}$.
 This property
 is the
centerpiece of the method we propose in the present paper.

\begin{propo} \label{integrality}
Continuing with the notation from Lemma~\ref{cong}, let $P(X) = \sum_{i=0}^n c_i X^i \in O_F[X]$.
For $v \in M_F^{\mathrm{fin}}$ a non-Archimedean valuation of $F$, consider the absolute
value $|\cdot|_v$ normalized so that $|\ell|_v = \ell^{-1}$ for the unique rational prime $\ell$
dividing $v$.  Then the  $X = 0$ Taylor series expansion of
  \begin{equation} \label{intseries}
  \sqrt[p]{P_p( X)^{p-1}P_{p^2}^{\kappa}( X)}
  \end{equation}
  in case (i) is convergent on the $v$-adic disk $|X|_v < \big( |c_0|_v / \max_{i=0}^n |c_i|_v \big)^{p^2}$,
   and the $X = 0$ Taylor series expansion of
    \begin{equation} \label{intseriesram}
  \sqrt[p]{P_{p^2}(  X)^{p-1}P_{p^2}^{\kappa}( X)}
  \end{equation}
  in case (ii) is convergent on the $v$-adic disk $|X|_v < \big( |a_0|_v / \max_{i=0}^n |a_i|_v \big)^{p^3}$.

\end{propo}

\begin{proof}
  Observe that, by expanding Newton's binomial,
  \begin{equation} \label{binom}
  (1 + p^2 Y)^{1/p} = \sum_{j \geq 0} \binom{1/p}{j}p^{2j} Y^j \in \Z[[Y]],
  \end{equation}
  where the $p$-adic valuation in $\binom{1/p}{j}$ is $-j - \mathrm{val}_p(j!) \geq -2j$.
  By Lemma~\ref{cong}, the polynomial $P_{p^e}(X)^{p-1}P_{p^2}(X)^{\kappa}$
  is congruent mod $p^2$ to the perfect $p$-th power $P_{p^e}(X)^{p}$, where $e = 1$ in case (i)
  and $e = 2$ in case (ii). This
  means that there is an $H \in X O_F[[X]]$ such that
  $P_{p^e}(X)^{p-1}P_{p^2}^{\kappa}(X)  = P_{p^e}(X)^p + p^2 H(X)$.

 Let $\C_v$ denote the completion of an algebraic closure of the field $F_v$, with
 $O_{\C_v} := \{ z \in \C_v \mid |z|_v \leq 1 \}$ its valuation ring.
   For
   any $c \in \C_v \setminus \{0\}$, by our normalization that $|\ell|_v = 1/\ell$ for
   the unique rational prime $\ell \mid v$,
    we  have $|c|_v \in \C_v^{\times}$ with $| |c|_v |_v = 1 / |c|_v$.
     For any polynomial $Q(X) = \sum_{i=0}^n b_i X^i \in O_{\C_v}[X]$ having $b_0 = Q(0) \neq 0$, we have by a geometric series expansion:
 $$
  b_0 |Q|_v \big/  Q\big( b_0 |Q|_v X \big) \in O_{\C_v}[[X]],
  $$
  where $|Q|_v := \max_{i=0}^n |b_i|_v$ is the Gauss norm.

  Our statement being trivial when $c_0 = 0$, apply the latter property to
  the polynomial $Q(X) := P_{p^e}( X) \in O_{\C_v}[X]$, having $Q(0) = c_0^{p^e} \in
  \{c_0^p, c_0^{p^2}\}$. We get an integral expansion
  $$
  c_0^{p^{e+1}} |P|_v^{p^{e+1}} \cdot P_{p^e}(c_0^{p^{e}} |P|_v^{p^e} \cdot X)^{-p} \in O_{\C_v}[[X]].
  $$
  Here, we have used Gauss's lemma (cf. Bombieri-Gubler, Lemma~1.6.3), yielding
  $|P_{p^e}|_v = |P|_v^{p^e}$.
  \emph{A fortiori}, as $c_0 |P|_v \in O_{\C_v}$, it follows
  \begin{equation} \label{inner}
  c_0^{p^{e+1}}  |P|_v^{p^{e+1}} \cdot P_{p^e}(c_0^{p^{e+1}} |P|_v^{p^{e+1}} \cdot X)^{-p} \in O_{\C_v}[[X]].
  \end{equation}
  Thus
  \begin{eqnarray*}
  \sqrt[p]{P_{p^e}(c_0^{p^{e+1}}|P|_v^{p^{e+1}} X)^{p-1}P_{p^2}^{\kappa}(c_0^{p^{e+1}}|P|_v^{p^{e+1}}  X)}\\
   = P_{p^e}(c_0^{p^{e+1}}|P|_v^{p^{e+1}}  X) \sqrt[p]{ 1 + p^2 H(c_0^{p^{e+1}}|P|_v^{p^{e+1}}  X)P_{p^e}(c_0^{p^{e+1}}|P|_v^{p^{e+1}}  X)^{-p} }  \\ \in O_{\C_v}[[X]],
  \end{eqnarray*}
  upon applying (\ref{binom}) to the substitution
  $$
  Y(X) = H(c_0^{p^{e+1}}|P|_v^{p^{e+1}} \cdot X) \big/ P_{p^e}( c_0^{p^{e+1}} |P|_v^{e+1} \cdot X)^p \in XO_{\C_v}[[X]].
  $$
  The integrality of the last expansion comes from (\ref{inner}) on noting that $H(X) \in X O_F[[X]]$.

  The result follows now from the integrality of this expansion and
   $\big|c_0^{p^{e+1}} |P|_v^{e+1} \big|_v = ( |c_0|_v / |P|_v )^{p^{e+1}} = ( |c_0|_v / \max_{i=0}^n |c_i|_v )^{p^{e+1}}$.
\end{proof}

The following is the generalization of Lemma~\ref{square}. It handles the ramified case (i); here and in the sequel, $\kappa$ is the Amoroso-Dvornicich element from Lemma~\ref{cong}. As $\mu_{\infty} \subset O_p$,
the cyclotomics split in linear factors over $O_p$, hence they are excluded
as soon as we suppose our polynomial $P \in O_p[X]$ is irreducible and non-linear.

\begin{lemma} \label{prad}
  Suppose that the polynomial $P(X) \in O_p[X]$
 is irreducible over $O_p$, of degree $> 1$, and with leading coefficient $a_0$.
 Suppose also that the polynomial $P_p(X)$ is a perfect $p$-th power in $O_p[X]$.
  Then the following are equivalent:
  \begin{itemize}
    \item[(i)] $a_0^{-p^2}P_{p^2} = a_0^{-p}P_p$;
    \item[(ii)] $a_0^{-p^2}P_{p^2} = \eta(a_0^{-p}) P_p^{\eta}$ for some $\eta \in \mathrm{Gal}(F/\Q)$;
    \item[(iii)] the polynomial $P_p(X)^{p-1}P_{p^2}^{\kappa}(X)$ is a perfect $p$-th power.
    \item[(iv)] the function $\sqrt[p]{P_p(X)^{p-1}P_{p^2}^{\kappa}(X)}$ is rational.
    \end{itemize}
  \end{lemma}

\begin{proof}
It is essentially the same as the proof of Lemma~\ref{square}. Suppose that $P_p(X)^{p-1}P_{p^2}^{\eta}(X)$ is a perfect $p$-th power for some
$\eta \in \mathrm{Gal}(F/\Q)$,
and take any root $\alpha$ of our irreducible polynomial $P(X)$. Since we assume the polynomial $P_p(X)$
to be irreducible over $F$,  $\alpha^p$ has to be among the roots $\sigma(\alpha)^{p^2}$
of $N_{F/\Q}(P_{p^2}(X))$, with $\sigma$ ranging over $\mathrm{Gal}(F/\Q)$. Letting $k$ the order of $\sigma$, we get
by a $k$-fold iteration:
\begin{eqnarray*}
\alpha^p = \sigma(\alpha^p)^{p} = \sigma^2(\alpha)^{p^3} = \sigma^2( \alpha^p)^{p^2}
= \cdots = \sigma^k(\alpha)^{p^{k+1}} = \alpha^{p^{k+1}}.
\end{eqnarray*}
Hence $\alpha$ is a root of unity.

An alternative argument, based on the theory of the Weil height, is to note that $\alpha^p = \sigma(\alpha^{p^2})$ implies
that $ph(\alpha) = h(\alpha^p) = h(\alpha^{p^2}) = p^2 h(\alpha)$, hence $h(\alpha) = 0$, and once again $\alpha$
is a root of unity.
\end{proof}

For the ramified case (ii), we will need to analyze when $P_{p^2}$
is a perfect $p$-th power.
 Two basic examples where this occurs are
  $P(X) = X^{p^2} - 2$, with $P_p(X) = (X^p - 2)^p$ and $P_{p^2}(X) = (X-2)^{p^2}$,
and $P(X) = X^p - 2$, with $P_p(X) = (X-2)^p$ and $P_{p^2}(X) = (X-2^p)^p$.
In both, already $P_p$ is a perfect $p$-th power.
Our next lemma shows that this feature is general, when $P$ is irreducible over $O_p$.

\begin{lemma}  \label{already}
 If $P \in O_p[X]$ is irreducible (over $O_p$) and $P_{p^2}$ is a perfect $p$-th power, then already $P_p$ is a
 perfect $p$-th power.
\end{lemma}

\begin{proof}
Let $\alpha$ be a root of $P(X)$. As $\alpha^{p^2}$ is among the roots of $P_{p^2}$,
which we assume a perfect $p$-th power, there must be a non-trivial
root of unity $\lambda \in \mu_{p^2} \setminus \{1\}$, of order either $p$
or $p^2$, such that $\lambda \alpha$ is also
among the roots of $P(X)$. As $P \in O_p[X]$ is irreducible, and $\lambda \in O_p$,
it follows that $\lambda^j \alpha$ is a root of $P$ for every $j \in \N$: there
is an element $\sigma$ in the Galois group of $P$ over $O_p$ carrying
$\sigma(\alpha) = \lambda \alpha$ (the Galois group of an irreducible
polynomial acts transitively on the roots), and then
$\sigma(\lambda \alpha) = \lambda \sigma(\alpha) = \lambda^2 \alpha,
\ldots, \sigma(\lambda^{j-1} \alpha) = \lambda^{j-1} \sigma(\alpha) =
\lambda^j \alpha$ are also roots. Thus $\alpha^p$
is a $p$-fold root of $P_p(X)$, and $P_p(X)$
is a perfect $p$-th power.
\end{proof}

\section{The Archimedean input: hedgehogs in the Riemann sphere and Dubinin's theorem} \label{hedgehogs}

By a \emph{hedgehog} with vertices $a_1, \ldots, a_n \in \C^{\times}$ we mean the union of the $n$ closed
radial segments $[0,a_i]$ joining the origin $0$ to the points $a_i$ in the complex plane.
Let us denote this hedgehog by $\mathcal{K}(a_1, \ldots, a_n) \subset \C$.
The simplest way to supply the Archimedean ingredient in our proof, also
leading to the constant $2^{1/4}$ in Theorem~\ref{original},
is to use the following theorem of Dubinin~\cite{dubinin,dubininbook}. It is itself
a byproduct of Dubinin's solution of a problem of Gonchar concerning maximal
harmonic measure at the origin for a set consisting of $n$ inward radial slits
of the unit disk $\{z : |z| < 1\}$ cutting to a common fixed length $r \in (0,1)$ from the origin.
The author is grateful to Alexander Er\"emenko and Fedor Nazarov
for pointing him to Dubinin's work.

\begin{thm}[Dubinin] \label{dubinin}
  The hedgehog $\mathcal{K}(a_1, \ldots, a_n) \subset \C$ has transfinite diameter
  at most $\big( \max_i |a_i|^n  / 4 \big)^{1/n}$, with equality if and only if
the points $a_1, \ldots, a_n$ form the vertices of a regular $n$-gon centered
at the origin.
\end{thm}

\begin{proof}
This is tantamount to the $r \to 0$ limit case of the {\it Corollary} (to Theorem~1 of {\it loc.cit.}) on page~270
of Dubinin's paper~\cite{dubinin}. Recall that the transfinite diameter of a compact set
$\mathcal{K} \subset \C$ is equal to the logarithmic capacity of $\mathcal{K}$ with respect
to the divisor $[\infty]$ on the Riemann sphere $\widehat{\C}$. If $\widehat{\C} \setminus \mathcal{K}$ is
simply connected (as is the case with our hedgehogs), the latter is furthermore equal to the inverse of the conformal
mapping radius of  $\widehat{\C} \setminus \mathcal{K}$ from the point $\infty$. It remains
to note the monotonicity of transfinite diameter under inclusion,
and that the transfinite diameter of the rotationally symmetric hedgehog
with $n$ spikes of a common length $b^{1/n}$ is equal to $(b/4)^{1/n}$. (A line segment of
length $b$ has transfinite diameter $b / 4$, and our rotationally symmetric hedgehog
is just the inverse image of the segment $[0,b]$ under the map $z \mapsto z^n$. One may
also see this value by observing that the map $z^{-1} \mapsto z^{-1} / (  1 - bz^{-n} / 4 )^{2/n} = z^{-1}
+ O(z^{-2})$ is a conformal isomorphism from the disk $|z^{-1}| < (4/b)^{1/n}$ about $\infty$ onto the complement
$\widehat{\C} \setminus \mathcal{K}(b^{1/n}e^{2\pi i h / n} : \, 0 \leq h < n)$.)

For alternative
approaches, we refer to the discussion in~\cite{mo}, based around the paper~\cite{konyaginlev}
of Konyagin and Lev. Dubinin's desymmetrization method, and the proof of Dubinin's theorem,
is exposed in detail in section~4.4 of the book~\cite{dubininbook}, and a related
 treatment can be also found in Baernstein~\cite{baernstein}
or Haliste~\cite{haliste}.
\end{proof}

\section{Gluing it together by arithmetic algebraization: Bertrandias's theorem}

The simplest route to concluding the proofs of Theorems~\ref{original} and~\ref{sz} is by applying
Bertrandias's rationality theorem~\cite{bertrandias,amice,cantor}. We recall here the relevant special case.
As is customary, we denote by $\widehat{\C}$ the Riemann sphere and by $d(A)$ the transfinite diameter of a compact subset $A \subset \C$. We
refer to Amice~\cite{amice}, section~5.4.1 for the definition and basic properties
of transfinite diameter. We note that Andr\'e's rationality theorem VIII~1.6 in~\cite{andre} (cf.
the proof of Theorem~\ref{holonomic} below) may be also
used as an alternative for Bertrandias's theorem, at the expense of worse numerical constants in Theorems~\ref{original}
and~\ref{sz}. For simplicity, since this suffices for our application, we only consider round circular domains at the
non-Archimedean places. Here, as in Proposition~\ref{integrality}, the $v$-adic absolute values follow the standard normalization
$|\ell|_v = 1/\ell$ for the unique rational prime dividing $\ell \mid v$.

 For the proof of Theorem~\ref{original}, the reader may simply take
Corollary~\ref{rationalinteger} below (the case $O_{F,S} = \Z$ in the following).

\begin{thm}[Bertrandias~\cite{bertrandias,amice,cantor}]  \label{rationality}
  Let $F$ be a number field and $S \subset M_F^{\mathrm{fin}}$ a finite
  subset of its non-Archimedean places.  Consider a formal power series
  $$
  f(1/X) = \sum_{n \geq 0} a_n X^{-n} \in O_{F,S}[[1/X]]
  $$
  with $S$-integral coefficients $a_n \in O_{F,S}$. For every $v \in S$,
  suppose that $f(1/X)$ is convergent on the $v$-adic disk $|1/X|_v < R_v$. Further, for every field embedding $\sigma : F \hookrightarrow \C$,
  consider a compact subset $\mathcal{K}_{\sigma} \subset \C$  such that $\widehat{\C} \setminus \mathcal{K}_{\sigma}$ is connected
   and the power series
  $f^{\sigma}(1/X):= \sum \sigma(a_n)  X^{-n}$
  has analytic continuation to $\widehat{\C} \setminus \mathcal{K}_{\sigma}$.  If
  \begin{equation} \label{capacity}
  \prod_{\sigma : F \hookrightarrow \C} d( \mathcal{K}_{\sigma}) < \prod_{v \in S} R_v,
  \end{equation}
  then $f \in F(X) \cap F[[1/X]]$ is a rational power series.

  The right-hand side of (\ref{capacity}) is understood to be $= 1$ in the case $S = \emptyset$
  of an empty product (a power series in $O_F[[1/X]]$).
\end{thm}

{\it An outline of the proof. } For the reader's convenience, we include some indications on the proof of Bertrandias's theorem.
As will be apparent from this outline, it is not necessary to have $\mathcal{K}_{\sigma} = \mathcal{K}_{\overline{\sigma}}$ for
complex conjugate embeddings $\sigma$ and $\overline{\sigma}$.

The argument, exposed in detail by Amice~\cite{amice}, is based on Kronecker's criterion
translating rationality of $f$ to the vanishing of the determinants $D_k(f) := \det \big( a_{i+j+1} \big)_{i,j=0}^k \in O_{F,S}$
for all large enough~$k$. The result is a majorization of $|N_{F/\Q}(D_k(f))|$, which is a rational $S$-integer 
 of denominator bounded by
$\prod_{v \in S}R_v^{-k^2 + o(k^2)}$.  For any sequence $p_0(z), p_1(z), \ldots$ of complex monic polynomials
with $\deg{p_j} = j$, and complex embedding $\sigma : F \hookrightarrow \C$, the
determinant
 $D_k(f^{\sigma}) = \det \big( \mathrm{res}_{z = 0}( f^{\sigma}(1/z)p_i(z)p_j(z) \, dz ) \big)_{i,j = 0}^k$.
The polynomials $p_j(z)$ are chosen according to Chebyshev's minimax problem on $\mathcal{K}_{\sigma}$:
there exists such a sequence $(p_j)_{j \in \N}$ satisfying
$\sup_{\mathcal{K}_{\sigma}(\epsilon)} |p_j| < (d(\mathcal{K}_{\sigma}) + o_{\epsilon \to 0, j \to \infty}(1))^j$,
on the $\epsilon$-thickening $\mathcal{K}_{\sigma}(\epsilon)$ of $\mathcal{K}_{\sigma}$. With this
choice of the monic polynomials $p_j(z)$ at $\sigma$, one applies the Cauchy residue formula and
Hadamard's volume inequality to bound the absolute value of the latter determinant by
the product of the $L^2$-norms of its rows. It results
that
$$
\limsup_{k \in \N} \Big|\sigma \Big(\det \big( a_{i+j+1} \big)_{i,j=0}^k \Big)\Big|^{\frac{1}{k^2}} \leq d(\mathcal{K}_{\sigma}),
$$
for each $\sigma : F \hookrightarrow \C$. \proofend

We spell out the special case of $O_{F,S} = \Z$, which suffices for Theorem~\ref{original}.

\begin{corol} \label{rationalinteger}
  An integer coefficients formal power series $f(1/X) \in \Z[[1/X]]$ is rational as soon as
  $f(1/X)$ is analytic on a connected complement $\widehat{\C} \setminus \mathcal{K}$
  of some compact set $\mathcal{K} \subset \C$ having transfinite diameter $d(\mathcal{K}) < 1$. \proofend
\end{corol}

\medskip

We have now assembled all the pieces necessary to prove our main result.

\medskip

{\it Proof of Theorem~\ref{original}. } As a preliminary remark, let us
start by noting that it is no loss of generality for Theorem~\ref{original} to
assume that our polynomial is reciprocal. Indeed, by Smyth's theorem~\cite{smythnr}, a non-reciprocal integer polynomial
of degree $n > 1$ has at least one root outside of the disk $|z| < \rho^{1/n}$,
where $\rho = 1.324\ldots > 1.189\ldots = 2^{1/4}$ is the real root of $\rho^3 = \rho + 1$.
Hence the reader may assume at this point that our monic integer polynomial $P$ is reciprocal,
and take $P^{\star} = P$ in the following.
This reduction however is not necessary, and with an eye to the generalization for
Theorem~\ref{sz} below, we shall denote $P^{\star}(X) = X^n P(1/X)$
the reciprocal polynomial. It has complex factorization $P^{\star}(X) = \prod_{i=1}^n (1 - \alpha_iX)$
and associated polynomials $P_m^{\star}(X) := \prod_{i=1}^n (1 - \alpha_i^{m}X) = (P_m)^{\star}$. Our
monicity convention for $P$ now translates into the free term condition $P^{\star}(0) = 1$,
so that Proposition~\ref{squarerad} is applicable to $P^{\star}$.

Consider the power series
$$
f(1/X) := \sqrt{P_2^{\star}(1/X)P_4^{\star}(1/X)} \in \Z[[1/X]],
$$
where the integrality of coefficients comes from Proposition~\ref{squarerad}.
We apply Bertrandias's Theorem (Corollary~\ref{rationalinteger}) to this power series
$f(1/X)$ and the hedgehog
$$
\mathcal{K} := \mathcal{K}( \alpha_i^2, \alpha_i^4 : P(\alpha_i) = 0 )
= \mathcal{K}( \alpha : P_2^{\star}(1/\alpha) P_4^{\star}(1 / \alpha) = 0 ) \subset \C
$$
of $2n$ vertices (possibily with overlapping edges) $\alpha_i^2, \alpha_i^4$, where
$\alpha_1, \ldots, \alpha_n$ are the complex roots of $P(X)$. By Dubinin's Theorem~\ref{dubinin}, this hedgehog
has transfinite diameter bounded above by
$$
d(\mathcal{K}) \leq \Big( \frac{\max_i |\alpha_i|^{4 \cdot 2n}}{4} \Big)^{\frac{1}{2n}} =
\Big( \frac{\max_i |\alpha_i|^{8n}}{4} \Big)^{\frac{1}{2n}}.
$$
Suppose now that $\max_i |\alpha_i| < 2^{\frac{1}{4n}}$. Then $d(\mathcal{K}) < 1$, and
since the power series $f(1/X)$ has analytic continuation to the connected and simply connected domain
$\widehat{\C} \setminus \mathcal{K}$ in the Riemann sphere avoiding the branch points of the square root,
 Bertrandias's Corollary~\ref{rationalinteger} yields the rationality of the power series
$f(1/X)$. This means that the polynomial $P_2P_4$ is a perfect square.
If the polynomial $P_2(X) \in \Z[X]$ is irreducible, this is only possible when $P_4 = P_2$;
Lemma~\ref{square} tells us that this is equivalent to $P$ being a cyclotomic polynomial of an odd level. If to the contrary
$P_2$ is reducible,
the minimal polynomial $Q(X)$ of $\alpha_1^2$  has degree $\deg{Q} = \deg{P} / 2 = n /2$,
and its roots are at the set $\{\alpha_1^2, \ldots, \alpha_n^2\}$. Upon passing in this
way from $P$ to $Q$,
Theorem~\ref{original} now follows by induction on the degree $n$.  \proofend

\medskip

{\it Proof of Theorem~\ref{sz}.} It follows similarly from Bertrandias's theorem over the Galois closure $F$
of the number field $K$. Clearing denominators, we may assume that our irreducible polynomial $P \in O_p[X]$ has coefficients
in $O_p$.
With $P^{\star}(X) := X^n P(1/X)$ denoting again the reciprocal, so that $P^{\star}(0) = a_0$ is the
leading coefficient of $P(X)$, we consider
at every complex embedding $\sigma : F \hookrightarrow \C$ the hedgehog
\begin{eqnarray*}
\mathcal{K}_{\sigma}
 :=  \mathcal{K}\big( \alpha : ( P_{p^e}^{\sigma} \cdot P_{p^2}^{\sigma\kappa} )(\alpha) = 0  \big)
\end{eqnarray*}
with $m \leq 2n$  spikes, where $e = 1$ or $2$ according to the cases (i) or (ii) in Lemma~\ref{cong},
and $\kappa \in \mathrm{Gal}(F/\Q)$  the Amoroso-Dvornicich element in Lemma~\ref{cong}.
Consider now the
algebraic power series
$$
f(1/X) := \sqrt[p]{ P^{\star}_{p^e}(1/X)^{p-1} (P_{p^2}^{\kappa})^{\star}(1/X)  } \in O_{F,S}[[1/X]],
$$
where $S \subset M_F^{\mathrm{fin}}$ is the set of non-Archimedean places of $F$ dividing $a_0 = P^{\star}(0)$.
The $S$-integrality of this expansion comes from Proposition~\ref{integrality}.

By Dubinin's theorem, our hedgehog $\mathcal{K}_{\sigma}$ has transfinite diameter
\begin{eqnarray*}
d(\mathcal{K}_{\sigma}) \leq \Big( \frac{   \max_{\alpha:  \, (P^{\sigma} \cdot P^{\sigma\kappa} )(\alpha) = 0} \max( | \alpha|^{p^2 \cdot  m}, |\alpha|^{p \cdot m}
 ) }{4} \Big)^{\frac{1}{m}} \\
 \leq   \frac{  \max_{\alpha : \,(P^{\sigma} \cdot P^{\sigma\kappa} )(\alpha) = 0} \max( | \alpha|, 1
 )^{p^2} }{2^{1/n}},
\end{eqnarray*}
as our number of spikes $m \leq 2n$.
Multiplying these estimates over all complex field embeddings $\sigma : F \hookrightarrow \C$, each polynomial
$\sigma(P) = P^{\sigma}$ occurs twice in the total product. By Proposition~\ref{integrality},
we furthermore know that $f(1/X)$ is convergent on the $v$-adic disk $|1/X|_v < R_v$ of the radius
$R_v := ( |a_0|_v / \max_{i=0}^n |a_i|_v )^{p^{e+1}}$, for every $v \in S$. Here $P(X) = a_0X^n + \cdots + a_n$, and
we have
$$
c_{P/F} =  \prod_{v \in M_F^{\mathrm{fin}}} \frac{\max_{i=0}^n |a_i|_v}{|a_0|_v} = \prod_{v \in S} \frac{\max_{i=0}^n |a_i|_v}{|a_0|_v} \in \N.
$$
Thus $\prod_{v \in S} R_v = c_{P/F}^{-p^{e+1}}$ and, noting that $p^{e+1} \leq p^3$ in any case, we find
 that the condition~(\ref{capacity}) in Bertrandias's theorem is met as soon as the $\sigma$-average
\begin{equation} \label{cond}
\frac{p^3}{[F:\Q]}\log{c_{P/F}}  + \frac{2p^2}{[F:\Q]} \sum_{ \sigma : F \hookrightarrow \C }  \max_{\alpha: \, \sigma(P)(\alpha) = 0} \log^+ {|\alpha|} < \frac{\log{2}}{n}.
\end{equation}
If the coefficients of $P$ are already contained by a subfield $K \subset F$,
the left-hand side of this inequality is unchanged when $F$ is replaced by $K$.

Consider now in turn the two alternatives (i) ($p$ unramified in $F$) and (ii) ($p$ ramified in $F$)
of Lemma~\ref{cong}. In the former case, by Lemma~\ref{square}, the conclusion $\sqrt[p]{ P_p^{p-1}P_{p^2}^{\kappa} } \in \Q(X)$
from Bertrandias's theorem (note that $n > 1$ and we assume $P$ to be irreducible over $O_p \supset \mu_{\infty}$) means  that
 $p \mid n$ and $P_p(X) = Q(X)^p$ for some polynomial $Q(X) \in O_p[X]$
of degree $n/p$.
 Then the first clause of Theorem~\ref{sz} follows inductively on passing to the irreducible factors
 of the degree-$n/p$
polynomial~$Q$, whose roots are at the set $\{ \alpha^p \mid P(\alpha) = 0 \}$.

If instead $p$ ramifies in the number field $F$, so that we are under the case (ii) of Lemma~\ref{cong}
and Proposition~\ref{integrality}, the conclusion from Bertrandias's theorem
is that either $P_{p^2}$ and $P_{p^2}^{\kappa}$ are scalar-proportional, or else $p \mid n$
and the polynomial $P_{p^2}(X)$ is a perfect $p$-th power. In the latter alternative,
Lemma~\ref{already} tells us that already $P_p(X) = Q(X)^p$ is a perfect
$p$-th power and, similarly to the previous, we get the first clause of Theorem~\ref{sz}
by induction on the degree, on passing again to the irreducible constituents
of the degree-$n/p$ polynomial $Q \in O_p[X]$.

 To conclude the proof of the first clause of Theorem~\ref{sz},
 it remains to consider the former alternative, that $a_0^{-p^2} P_{p^2} = (a_0^{-p^2} P_{p^2})^{\kappa}$. In
 that case,
 Lemma~\ref{cong}  supplies the existence of a root of unity $\zeta \in \mu_{Np}$
 such that
   $(a_0\zeta^n)^{-1} P(\zeta X) \in (C_{N/p} \cdot \Qtp) [X]$. Now we conclude by
 a level-lowering induction upon replacing the polynomial $P(X) \in (C_N \cdot \Qtp)[X]$
by the polynomial
$$
\widetilde{P}(X) :=  (a_0\zeta^n)^{-1} \cdot P(\zeta X) \in ((C_{N/p} \cdot \Qtp) [X],
$$
which remains irreducible over $O_p$ and has $c_{\widetilde{P}/K} = c_{P/K}$.

In the case ($p \nmid N$) that $p$ is unramified in $F$, we have $e = 1$ in the previous,
and the condition~(\ref{cond}) for Bertrandias's theorem naturally
strengthens to a $p^2$ coefficient instead of $p^3$ in $\log{c_{F/K}}$, leading
to the second clause of Theorem~\ref{sz}.

For the third clause, consider finally the case that $P \in \Qtp[X]$,
so that we have $e = 1$ and $\kappa = \mathrm{id}$ in the previous. The estimate
on the transfinite diameter $d(\mathcal{K}_{\sigma})$ of our hedgehog
then improves as follows. If $\sigma : F \hookrightarrow \C$ is such that
$\sigma(P) = P^{\sigma} \in \C[X]$ has a root outside of the unit disk $|X| \leq 1$,
then $\max_{\alpha: \, P^{\sigma}(\alpha) = 0} \max(|\alpha|^{p^2}, |\alpha|^p)
= \max_{\alpha: \, P^{\sigma}(\alpha) = 0} |\alpha|^{p^2}$. If to the contrary
all the roots of $P^{\sigma} \in \C[X]$ are contained by the disk $|X| \leq 1$,
then $\max_{\alpha: \, P^{\sigma}(\alpha) = 0} \max(|\alpha|^{p^2}, |\alpha|^p)
= \max_{\alpha: \, P^{\sigma}(\alpha) = 0} |\alpha|^{p}$. Unifying, we have
\begin{eqnarray*}
\max_{\alpha: \, P^{\sigma}(\alpha) = 0} \max(|\alpha|^{p^2}, |\alpha|^p)
= \max_{\alpha: \, P^{\sigma}(\alpha) = 0} \big\{  \max(|\alpha|,1)^{p^2} \cdot \max(|1/\alpha|,1)^{-p} \big\}
\end{eqnarray*}
in either case, leading to the improved estimate
$$
d(\mathcal{K}_{\sigma}) \leq   \frac{  \max_{\alpha : \,P^{\sigma}(\alpha) = 0}  \big\{  \max(|\alpha|,1)^{p^2} \cdot \max(|1/\alpha|,1)^{-p} \big\} }{2^{1/n}},
$$
and to the third clause of Theorem~\ref{sz} on totally $p$-adic coefficients.
 \proofend

\section{Some further results and questions}

In this final section we indicate a further improvement on our basic constant $2^{1/4}$ in
the reciprocal and atoral case of
Theorem~\ref{original}, an extension of Theorem~\ref{original} to heights of holonomic ($D$-finite) power series,
 and an application of Theorem~\ref{sz} raising the possibility
of a Lehmer problem relative over $\Q^{\mathrm{ab}} \cdot \Q^{\mathrm{t.}p}$.

\subsection{} The following becomes of interest upon noting that the best available constant (Dubickas \cite{dubickasnr,dubickasseq})
to date on the non-reciprocal case of the Schinzel-Zassenhaus problem
is $0.3096\ldots < 0.346 < (\log{2}) / 2$. In the case of a reciprocal
polynomial without roots on the unit circle,
we obtain the stronger constant $(\log{2}) / 2$.

\begin{thm} \label{atoral}
  If $P \in \Z[X]$ is a monic reciprocal polynomial of degree $n > 1$ and non-vanishing
  on the unit circle,
 then $P$ has a root $\alpha$ with
$|\alpha| \geq 2^{\frac{1}{2n}}$.
\end{thm}

\begin{proof}
Since $P$ is reciprocal of degree $n > 1$, and without roots on the unit circle, the degree $n = \deg{P}$ is even and the complex roots
of $P$ occur in $n/2$ radial pairs $(\alpha, 1 / \bar{\alpha})$ about the unit circle.
Thus the hedgehog $\mathcal{K}(\alpha^2, \alpha^4 : P(\alpha) = 0)$ used in the
proof of Theorem~\ref{original} has a total number of at most $n$, rather than $2n$, non-overlapping spikes.
Dubinin's theorem then bounds
$$
d(\mathcal{K}) \leq \Big( \frac{ \max_{\alpha: P(\alpha) = 0} |\alpha|^{4n}}{4} \Big)^{\frac{1}{n}},
$$
whence the constant $2^{1/4}$ of Theorem~\ref{original} improves to $2^{1/2}$ in this situation.
\end{proof}

\subsection{} In this section we compare the bounds in Theorems~\ref{original} and~\ref{atoral} to
previous results available in the literature. The strongest previous asymptotic result on
the Schinzel-Zassenhaus conjecture was the one obtained by Dubickas~\cite{dubickassz} with
Dobrowolski's method: 
$$
\max_{\alpha: \, P(\alpha) = 0} \log{|\alpha|} \geq \Big( \frac{64}{\pi^2} - o_{n \to \infty}(1) \Big) \frac{1}{n} \Big( \frac{\log{\log{n}}}{\log{n}} \Big)^3.
$$
Up to the numerical coefficient, this is of the same strength as Dobrowolski's general lower bound~\cite{dobrowolski2}
on Mahler measure. The latter remains the best known on the Lehmer problem in the case of Salem numbers, 
which is the opposite of the Schinzel-Zassenhaus case studied by our present paper. Dobrowolski's method~\cite{dobrowolski1,dobrowolski2}
is based on an extensive use of the mod $p$ congruences $P_p \equiv P \mod{p \cdot \Z[X]}$, for $P(X) \in \Z[X]$, 
which have since become a standard tool in the subject. 

On the other hand, for the small degrees $37 \leq n < 60$, the sharpest bound remains the one by Matveev~\cite{matveev}, 
with a precision by Rhin and Wu~\cite{rhinwu}: 
\begin{equation} \label{mrw}
\max_{\alpha: \, P(\alpha) = 0} \log{|\alpha|} \geq  \frac{3\log(n/2)}{n^2}.
\end{equation}
Matveev's bound (\ref{mrw}) also holds in the range $12 \leq n \leq 36$, 
but for $29 \leq n \leq 36$ it is superseded by the results of Flammang, Rhin
and Sac-\'Ep\'ee~\cite{numerical} determining the complete list of integer polynomials of 
degree up to $36$ and Mahler measure at most $1.324$, and for $n \leq 28$ it is 
superseded by the results of Rhin and Wu~\cite{rhinwu} determining the least positive value
of $\max_{\alpha: \, P(\alpha) = 0} \log{|\alpha|}$ in each degree $n \leq 28$. 

For $n \geq 59$, our Theorem~\ref{original} becomes stronger than Matveev's bound (\ref{mrw}).
In the ``atoral'' case of a reciprocal polynomial without roots on the unit circle, 
Theorem~\ref{atoral} improves upon (\ref{mrw}) already for $n \geq 20$. 

\subsection{} We also obtain a similar height lower bound for holonomic power series with rational (not necessarily $S$-integral)
coefficients. The following is in the style of Bell-Nguyen-Zannier~\cite{bnz}, and had been our starting point
in the present enquiry; it is the precise counterpart of Theorem~\ref{original}
in the irrational holonomic case. To make the connection, note indeed that
 $h(1/P) = \max_{\alpha : P(\alpha) = 0} \log{|1/\alpha|}$, whenever $P \in \Z[X]$
is an integer polynomial with $P(0) = 1$, and $h(f(X^m)) = h(f(X)) / m$,
where the number of singular points of $f(X^m)$ is $m$ times the number
of singular points of $f(X)$.

\begin{thm} \label{holonomic}
  For $f(X) = \sum_{n \geq 0} a_n x^n \in \Q[[X]] \setminus \{0\}$ a formal power series
  with rational coefficients, define its \emph{height} by
  $$
  h(f) := \limsup_{n \in \N} \frac{1}{n} h([a_0:\ldots:a_n]),
  $$
the upper exponential growth rate of the Weil heights of the
coefficients vectors
of the polynomial truncations of $f$. Suppose that
$f(X)$ satisfies a non-zero linear differential equation
$Lf = 0$, where
$$
L = q(X) \frac{\partial^r}{\partial X^r} + \sum_{i=0}^{r-1} a_i(X) \frac{\partial^i}{\partial X^i}, \quad q; a_0, \ldots, a_{r-1} \in \Q[X]
$$
is a linear differential operator with polynomial coefficients. Let $k$ denote the total number of
distinct roots of the leading polynomial $q(X)$, not counting multiplicities: the total number of (apparent)
singularities for the differential equation, not counting the possible singularity at $\infty$.

Then either $f(X)$ is a rational function of the special form
$f(X) =p(X) / (X^k-1)^m$, with $p(X) \in \Q[X]$ a polynomial; or else its height fulfills the uniform
lower bound
\begin{equation}  \label{uniformrat}
h(f) \geq \frac{1}{150k}.
\end{equation}
In the case that $f(X) \in \Z[[X]]$ has integer coefficients, but is not rational, the numerical coefficient in the lower bound (\ref{uniformrat})
furthermore strengthens to the best-possible value $\log{4}$:
\begin{equation}\label{uniformint}
  h(f) \geq \frac{\log{4}}{k}.
\end{equation}
\end{thm}

To see the sharpness of the value $\log{4}$ in the $\Z[[X]]$ case, observe that the key
example (\ref{2int}) of $f(X) := \sqrt{1-4X} \in \Z[[X]]$ has a single branch point 
at $X = 1/4$, and height $h(f) = \log{4}$.

\begin{proof}
The $\Z[[X]]$ case is an immediate combination of Dubinin's Theorem~\ref{dubinin},
Bertrandias's Corollary~\ref{rationalinteger} applied to $f(1/X)$ with $\mathcal{K}$ the 
hedgehog with vertices at the singular points of $f(1/X)$, and the obvious formula $h(f) = \log^+{(1/\rho)}$
whenever $f(X) \in \Z[[X]]$ has rational integer coefficients and 
Archimedean convergence radius $\rho$. 

For the general $\Q[[X]]$ case, we instead
use a rationality theorem of Andr\'e (\cite{andre}, VIII 1.6),
  and combine it with our Theorem~\ref{original} and a similar argument to its proof,
  based on Dubinin's theorem. The following particular case of Andr\'e's theorem results with the
  choices $\mu = 2, \nu = 1$; $\tau := 2$; $y_0 = 1, y_1 = f(X)$; $K = \Q$, $V = \{\infty\}$;
   $f_{1,\infty} = z, g_{\infty} = 1$; $\kappa_{\infty} = R$,
  so that $S := \exp(\chi_{\infty})$:
  \begin{quote}
    {\it The formal power series $f(X) \in \Q[[X]] \setminus \{0\}$
    is rational as soon as there exist two radii $S \geq R > 0$ satisfying
    $$
    \log{R} - \frac{3}{4}\log{S} > 3h(f),
    $$
     and a complex analytic
    function $x(z) : D(0,R) \to D(0,S)$ with $x(0) = 0$ and $x'(0) = 1$, such that
    the power series $f(x(z)) \in \C[[z]]$ is convergent on the disk $D(0,R)$.}
  \end{quote}
  We again apply Dubinin's theorem~\cite{dubinin} (the \emph{Corollary} on page~270).
  Let $\Sigma \subseteq \{ \alpha \in \C^{\times} \mid q(\alpha) = 0 \} \subset \C^{\times}$ be
  the set of finite singularities of $f$, not counting the possible singularity at $\infty \in \widehat{\C}$.
   As those are among the roots of the leading polynomial
  coefficient (``denominator'') $q$, we have $|\Sigma| \leq k$. With $S \geq \max( |\alpha| : \alpha \in \Sigma )$
   a parameter, consider the simply connected planar domain
  (a $k$-slit disk)
  $$
  U := D(0, S) \setminus \big\{ [1, S/|\alpha|] \cdot \alpha : \alpha \in \Sigma  \big\}.
  $$
  By Dubinin's theorem, the conformal mapping radius of $(U,0)$ is not smaller than
  the one for the rotationally symmetric domain
  $$
  D(0,S)  \setminus \big\{ [b, S] \cdot e^{2\pi i h / k} : 0 \leq h < k \big\}, \quad b := \min_{\alpha \in \Sigma}{|\alpha|},
  $$
  which in turn is precisely equal to
  $$
 R := \Big( \frac{4b^kS^{2k}}{(b^k+S^k)^2} \Big)^{1/k} = \frac{4^{\frac{1}{k}}bS^2}{(b^k+S^k)^{2/k}},
  $$
  by an explicit computation of the Riemann map of the slit circular disk
  $D(0,S^k) \setminus [b^k,S^k]$, for which we refer for example to
  Rumely~\cite{rumely}, (2.50) on page~18. Hence, with this choice of $R = R(S,b) = R(S, \min_{ \alpha \in \Sigma}{|\alpha|})$,
  there exists an analytic map $x(z) : D(0,R) \to D(0,S)$ such that $x(0) = 0$, $x'(0) = 1$, and
  the power series $f(x(z)) \in \C[[z]]$ is convergent on the disk $D(0,R) = \{z : |z| < R\}$.
 With the arbitrary parameter $S \geq b$, Andr\'e's theorem yields that either $f$ is rational,
 in which case the conclusion follows from Theorem~\ref{original}; or else it fulfills the height
 lower bound
 $$
 h(f) \geq \frac{\log{4}}{3k} +  \frac{1}{3}\log{b} - \frac{1}{4}\log{S} - \frac{2}{3k}\log{\big(1 +(b/S)^k  \big)}.
 $$
 Choosing $S := 2^{1/k}b$, and estimating
 $$
 \frac{\log{4}}{3} - \frac{\log{2}}{4} - \frac{2}{3}\log{(3/2)} = 0.018\ldots > \frac{1}{100},
 $$
 we get $h(f) \geq \frac{1}{100k} + \frac{1}{12}\log{b}$.
 We have in any case $h(f) \geq \log^+(1/b)$, on considering that $1/b$ is the Archimedean convergence
radius of $f(X)$, hence $h(f) \geq \frac{1}{100k} - \frac{1}{12}h(f)$, and finally
$h(f) \geq \frac{1}{150k}$.
\end{proof}

\subsection{} We give an application of Theorem~\ref{holonomic} to an upper estimate
of the least critical value for certain rational functions.
In the polynomial case it is subsumed by Dubinin's solution~\cite{dubinincritval}
of an extremal problem of Smale~\cite{smale}, but for rational functions this
type of inequality seems to be new. It would be interesting to clarify the 
cases where the equality holds (they include $R(z) = z \pm z^2$, and no further
polynomials thanks to Dubinin's bound~\cite{dubinincritval}), and to
 give a refinement about
rational functions in $\C(z)$, rather than $\Q(z)$. 

\begin{corol}  \label{criticalval}
  Let $R(z) \in \Q(z)$ be a rational function of degree $n > 1$ and having
  $R(0) = 0$ and $R'(0) \neq 0$. Write $R(z) = P(z) / Q(z)$ with $P(z), Q(z) \in \Z[z]$
  integer polynomials, and assume that $P'(0) = 1$. Then the least critical value of 
  $R(z)$ satisfies the upper bound
   \begin{equation} \label{criticalineq}
  \min \big\{ |R(w)| : R'(w) =0 \big\} \leq  \exp\Big( - \frac{\log{4}}{\deg{P}+\deg{Q}-1}\Big).
  \end{equation}
\end{corol}

\begin{proof}
  Consider the bivariate rational function
$$
F(X,Y) = \sum_{i,j \geq 0} a_{i,j}X^iY^j := \frac{1}{XQ(Y) - P(Y) / Y} \in \Z[[X,Y]],
$$
where the integrality property of the expansion comes from a geometric series 
expansion using our assumption
$P(Y) / Y \in 1 + Y\Z[[Y]]$.
Let $f(z) := \sum_{i \geq 0} a_{i,i} z^i$ be the \emph{diagonal} power series. It is algebraic
by the Cauchy integral formula
\begin{eqnarray*}
f(z) = \frac{1}{2\pi i} \int_{|y| = \epsilon} \frac{1}{(z/y)Q(y) - P(y) / y} \, \frac{dy}{y} \\
=  \frac{1}{2\pi i} \int_{|y| = \epsilon} \frac{Q(y)}{z - R(y)} \, dy,
\end{eqnarray*}
whose evaluation by residues shows
more precisely that $f(z)$ is algebraic of degree $n$ and with branch points
 among the critical value set $\{ R(w) : R'(w) = 0 \}$ of the rational function $R(z)$. 

As $f(z) \in \Z[[z]]$, the height $h(f) = \log^+{(1/\rho)}$, where $\rho$ is the Archimedean
convergence radius. By the remark on the branch points, the latter is of the form $\rho = |R(w)|$
for some critical value $R(w)$, $R'(w) = 0$, of $R$. 
As $f(z) \notin \Q(z)$ and has $\deg{P} + \deg{Q} - 1$ singularities (the critical values of the rational function
$R = P/Q$), Theorem~\ref{holonomic}
yields the minorization  $\log^+{1/|R(w)|} \geq \frac{\log{4}}{\deg{P}+\deg{Q}-1}$. 
 This is equivalent to (\ref{criticalineq}).
\end{proof}

\subsection{}
Our next result, an application of Theorem~\ref{sz}, raises the possibility of a relative Lehmer conjecture over totally $p$-adic fields.
The $\Q^{\mathrm{ab}}$-relative conjecture, and a Dobrowolski-strength result in this direction, was
discovered by Amoroso and Zannier~\cite{amorosozannier} soon-after the lead~\cite{amorosodvornicich} by Amoroso and Dvornicich.
The Bogomolov height gap property for the field $\Q^{\mathrm{ab}} \cdot \Q^{\mathrm{t.}p}$
is itself contained in Theorem~1.5 in Amoroso-David-Zannier~\cite{adz}. The following result,
and the question whether its exponent  ``$-2$'' may possibly be raised to the Lehmer-strength
exponent ``$-1$,'' is then very much in the spirit of~\cite{adz} and~\cite{amorosozannier}.

\begin{thm} \label{relheightbound}
  For $\alpha \in \mathbb{G}_m(\bar{\Q}) \setminus \mu_{\infty}$ a non-torsion algebraic point of the multiplicative group,
  and $p$ an arbitrary prime, the following lower bounds hold on the absolute canonical Weil height:
\begin{eqnarray*}\label{weilheight}
 h(\alpha) \geq \log{2} \cdot ( p \cdot [\Q^{\mathrm{t.}p}(\alpha) :   \Q^{\mathrm{t.}p}])^{-2}.
\\
  h(\alpha) \geq (\log{\sqrt{2}}) \cdot p^{-3} \cdot [\Q^{\mathrm{ab}} \cdot \Q^{\mathrm{t.}p}(\alpha) :  \Q^{\mathrm{ab}} \cdot \Q^{\mathrm{t.}p}]^{-2}.
\end{eqnarray*}
Moreover, if either $\alpha$ is an algebraic integer or $p$ is unramified in $\Q(\alpha)$, then
the latter bound strengthens to
$$
  h(\alpha) \geq (\log{\sqrt{2}}) \cdot (p \cdot [\Q^{\mathrm{ab}} \cdot \Q^{\mathrm{t.}p}(\alpha) :  \Q^{\mathrm{ab}} \cdot \Q^{\mathrm{t.}p}])^{-2}.
$$
\end{thm}

\begin{proof}
  Let
  \begin{eqnarray*}
  n := [\Q^{\mathrm{t.}p}(\alpha) :  \Q^{\mathrm{t.}p}], \quad \textrm{resp.} \\
  n :=[\Q^{\mathrm{ab}} \cdot \Q^{\mathrm{t.}p}(\alpha) :  \Q^{\mathrm{ab}} \cdot \Q^{\mathrm{t.}p}]
  \end{eqnarray*}
  so that $\alpha$ satisfies a degree-$n$ equation $P(\alpha) = 0$ over  some number field $K \subset \Qtp$, resp. $K \subset \Q^{\mathrm{ab}} \cdot \Q^{\mathrm{t.}p}$.
For every complex field embedding $\sigma : K \hookrightarrow \C$, we estimate in Theorem~\ref{sz} trivially
\begin{eqnarray*}
 \max_{\beta: \, \sigma(P)(\beta) = 0} \log^+{|\beta|}
\leq \sum_{\beta: \, \sigma(P)(\beta) = 0} \log^+{|\beta|}.
\end{eqnarray*}
It remains to note that
$$
h(\alpha) = \frac{m( F )}{[K:\Q]} = \frac{1}{[K:\Q]} \log{c_{P/K}} + \frac{1}{[K:\Q]} \sum_{\substack{ \sigma: K \hookrightarrow \C \\
\beta: \sigma(P)(\beta)  = 0 }} \log^+{|\beta|},
$$
where $F = \mathrm{const} \cdot N_{K/\Q}(P) \in \Z[X]$ is the minimal equation of $\alpha$ over $\Z$;
its leading coefficient is $c_P$ by definition.
The result now follows from Theorem~\ref{sz}.
\end{proof}



\subsection{Some open problems} In the reverse direction
to Theorem~\ref{relheightbound}, the example of the $n$-th radical from a root of the $\Q_p$-split equation $X^{p-1} - pX - 1 = 0$
shows that $h(\alpha)$ and
$$
\frac{1}{[K:\Q]} \sum_{\sigma : K \hookrightarrow \C} \max_{\alpha: \, \sigma(P)(\alpha) = 0} \log^+{|\alpha|}
$$
can both be as small as $\lesssim \frac{\log{p}}{p-1} \frac{1}{n}$ as $P$ ranges
over the irreducible degree-$n$ monic non-cyclotomic
polynomials over the ring of algebraic integers in $\Qtp$, for every $p$ and $n$,
where the implied coefficient is understood to approach $1$ in the $p \to \infty$
asymptotic, uniformly in $n$. This leaves something to be desired
in both the $p$ and $n$ aspects of Theorem~\ref{relheightbound}, and
indeed already in the $p$ aspect of Theorem~\ref{sz}.

While it is known (Pottmeyer~\cite{pott}) that
$$
\inf_{\alpha \in \mathbb{G}_m(\Qtp) \setminus \mu_{\infty}} h(\alpha) = (1 - o_{p \to \infty}(1))\frac{\log{p}}{p},
$$
to the best of our knowledge it is already an open question which of the basic $\ll \frac{\log{p}}{p}$ upper or $\gg \frac{\log{p}}{p^2}$
lower bounds  is closer to the truth on the quantities
$$
\inf_{ \substack{ \alpha \in \mathbb{G}_m(\bar{\Q}) \setminus  \mu_{\infty} \\ [\Qtp(\alpha):\Qtp] = 2 }}{ h(\alpha)}
\quad \textrm{and} \quad \liminf_{\alpha: \, [\Qtp(\alpha):\Qtp] = 2 } h(\alpha);
$$
the $\liminf$ here denoting the first accumulation point of the heights of algebraic numbers
fulfilling a quadratic equation over a totally $p$-adic number field.
One may even ask, extending a conjecture of Fili~\cite{fili} (possibly
also adding
several simultaneous splitting conditions), whether the basic lower bounds
\begin{eqnarray*}
\liminf_{\substack{ \alpha \in \bar{\Z} \textrm{ algebraic integer} \\ \, [\Qtp(\alpha):\Qtp] = n}}{ h(\alpha)} \geq \frac{1}{p-1} \frac{\log{p}}{n}
\end{eqnarray*}
and
$$
\liminf_{\substack{ \alpha \in \bar{\Q} \\ \, [\Qtp(\alpha):\Qtp] = n}}{ h(\alpha)} \geq \frac{p}{p^2-1} \frac{\log{p}}{n}
$$
on the respective first accumulation points
ought to be equalities.

Whether or not a totally real or CM analog should exist of Theorems~\ref{sz} and~\ref{relheightbound}, under the additional (and necessary)
proviso that the polynomial $N_{K/\Q}(P) \in \Z[X]$ has at least one root off the unit circle, remains completely open to the method.
Likewise for the apparently closely related generalization  of the Schinzel-Zassenhaus problem raised by Chowla
and Blanksby~\cite{blanksby},
who drew a motivation from Tur\'an's power sums method: {\it if a non-monic integer polynomial $qX^n + \cdots \in \Z[X]$ of degree $n$
and leading coefficient $q$
has a root outside of the closed unit disk $|X| \leq 1$, does it also have a root outside of the larger disk $|X| \leq 1 + c / (q n)$?}

Lastly, apart from the more routine task of extending Theorem~\ref{sz} relatively over
certain  infinite degree fields more general than $\Q^{\mathrm{ab}} \cdot \Qtp$ (such as the ones
studied for Property (B) by Amoroso, David and Zannier~\cite{adz}), it could be valuable to try to extend the method
higher dimensionally, the first interesting case being a $\mathbb{G}_m^2$ version of the Schinzel-Zassenhaus problem in the spirit of Amoroso and David~\cite{amorosodavid}: \emph{if $\alpha , \beta \in \bar{\Q}^{\times}$ are multiplicatively independent non-zero algebraic numbers,
and $n$ denotes the minimum degree of a non-zero algebraic relation $P(\alpha,\beta) = 0$ in $\Z[X,Y]$, is there some
complex field embedding $\sigma: \Q(\alpha,\beta) \hookrightarrow \C$ under which at least one of $\log{|\sigma(\alpha)|}$
or $\log{|\sigma(\beta)|}$ is $\gg 1/n$?} The Amoroso-David theorem~\cite{amorosodavid} includes a proof of the weaker $\gg_{\epsilon} n ^{-1-\epsilon}$
bound, for any $\epsilon > 0$. Whereas the rationality theorem has multivariate variants such as VIII 1.6 in Andr\'e~\cite{andre},
a major new difficulty here appears to be in interpreting the multiplicative independency hypothesis for $\alpha$ and $\beta$,
and finding something to take the place of our basic criterion for cyclotomy via the rationality of $\sqrt{P_2 P_4}$.

\end{document}